\documentclass[sn-mathphys-num]{sn-jnl}
\usepackage{amsfonts,amssymb,latexsym,amsthm, epsfig,amsmath,amscd, amsxtra, color, calligra,mathrsfs,comment,url, accents}
\usepackage{tikz-cd}
\usepackage{enumerate,verbatim,mathrsfs,comment}
\usepackage[all,2cell,ps]{xy}
\usepackage{mathtools} 
\usepackage{extpfeil}   
\usepackage{soul}  
\usepackage{tikz,tikz-cd,enumerate}
\usetikzlibrary{matrix,arrows,decorations.pathmorphing}
\usepackage{emptypage} 
\usepackage{floatrow} 
\usepackage{lineno} 
\usepackage{float}
\usepackage{xcolor} 
\usepackage{graphicx}
\usepackage{tikz,tikz-cd,tkz-graph,enumerate,enumitem}
\usepackage{tabularx}
\usetikzlibrary{positioning}
\usepackage[title]{appendix}%
\usepackage{textcomp}%
\usepackage{manyfoot}%
\usepackage{booktabs}%
\usepackage{algorithm}%
\usepackage{algorithmicx}%
\usepackage{algpseudocode}%
\usepackage{listings}%

\usepackage[symbols,nogroupskip,sort=none]{glossaries-extra}
\usepackage{enumitem}     

\newcommand{\mf}{\mathfrak}

\newcommand{\ol}{\ol}

\newcommand{\ra}{\rightarrow}

\newcommand{\mbb}{\mathbb}
\newcommand{\tn}{\textnormal}

\newcommand{\Max}{\text{Max}}
\newtheorem{de}{Definition}[section]
\newtheorem{re}[de]{Remark}
\newtheorem{pr}[de]{Proposition} 
\newtheorem{tr}[de]{Theorem}
\newtheorem{lm}[de]{Lemma}

\newtheorem{co}[de]{Corollary}

\newcommand{\MD}{\mbox{$\mathbb{D}$}}

\newcommand{\m}{\mathfrak m}

\newcommand{\simm}[1]{\underset{#1}{\sim}}

\newcommand{\pz}[1]{{\mathbb{Z}}_{+}^{#1}}
\newcommand{\nz}[1]{{\mathbb{Z}}^{#1}}

\def\vp{\rm \vspace{0.2cm}}
\def\M{\rm Max}
\def\m{\mf{m}}

\def\Um{\rm Um}
\def\GL{\rm GL}

\def\EO{\rm EO}

\def\E{\rm E}

\def\G{\rm G}

\def\Sp{\rm Sp}
\def\ko{\rm K_1O}
\def\ksp{\rm K_1Sp}
\def\k{\rm K_1}
\def\K{\rm K}

\def\O{\rm O}
\def\ESp{\rm ESp}
\def\EO{\rm EO}

\DeclareMathOperator{\Hom}{Hom}

\DeclareMathOperator{\rank}{rank}

\DeclareMathOperator{\hgt}{ht}
\DeclareMathOperator{\conv}{conv}
\DeclareMathOperator{\intt}{int}

\begin{document}
	\title[Elementary action of classical groups on unimodular rows over monoid rings]{Elementary action of classical groups on unimodular rows over monoid rings}
	
	
	\author[1]{\fnm{Rabeya} \sur{Basu}}\email{rbasu@iiserpune.ac.in}
	
	\author*[1]{\fnm{Maria} \sur{A. Mathew}}\email{maria.mathew@acads.iiserpune.ac.in}
	
	\affil[1]{\orgdiv{Department of Mathematics}, \orgname{Indian Institute of Science Education and Research (IISER) Pune}, \orgaddress{\city{Pune}, \postcode{411008}, \state{Maharashtra}, \country{India}}}


	\date{}
	\maketitle

	\noindent
	{\small Abstract: The elementary action of symplectic and orthogonal groups 
		on unimodular rows of length $2n$ is transitive for $2n\ge \max(4, d+2)$ in the symplectic case, and $2n\ge \max(6, 2d+4)$ in the orthogonal case, 
		over monoid rings $R[M]$, where $R$ is a commutative noetherian ring of dimension $d$, and 
		$M$ is commutative cancellative torsion free monoid. 
		As a consequence, one gets the surjective stabilization bound 
		for the ${\k}$ for classical groups. This is an extension of J. Gubeladze's results for linear groups.} \vp



	Throughout the paper we will assume: (i) all rings are noetherian and commutative with identity, (ii) all monoids are commutative, cancellative and torsion-free. 
	
	\section{Introduction}\label{intro}

	The theory of unimodular rows became a popular topic in the study of classical algebraic $\K$-theory after A. Suslin produced the matrix theoretic proof of the celebrated Quillen--Suslin theorem; the main  ingredient of Serre's problem for projective modules. We see lots of results involving the theory of action of elementary group on unimodular rows in the works of Murthy--Swan--Tower--Mohan Kumar--Roitman--Vaserstien--Suslin--Kopeiko--Vavilov--Stepanov--Bhatwadekar--Rao {\it et al.} Later generalizations came for the traditional classical groups and for modules over Laurent polynomial rings. Since monoid rings are a natural extension of Laurent polynomial rings, D.F. Anderson \cite{MR526663-Anderson} conjectured the freeness of finitely generated projective modules over normal monoid rings over a PID. The proof of  Anderson's conjecture by J. Gubeladze \cite{Gubeladze2-maximalMR937805} led to the discussion of the nature of the action of elementary groups on unimodular rows over monoid rings. In two consecutive papers \cite{Gubeladze-umodrow1-MR1161570} \cite{Gubeladze3-el2MR1206629}, published in 1989 and 1991, Gubeladze proved the transitive action of elementary group on unimodular rows over monoid rings under certain assumptions on the monoid. Finally, in \cite{Gubeladze-unimodularMR3853049}, he gave a proof for for a much general case. In this paper, we deduce the analogue of Gubeladze's result for the traditional classical groups; {\it viz.} for the symplectic and orthogonal groups. \vp
	
	Let us consider a ring $R$ of Krull dimension $d$.  We fix the notation $\MD(R)$ as
	\[ 
	\MD(R):= \left\{
	\begin{array}{ll}
		\max\{4, d+2\}, & \text{ in the symplectic case, } \\
		\max\{6, 2d+4\}, & \text{ in the orthogonal case. }\\
	\end{array} 
	\right. 
	\]
	
	\noindent The aim of this article is to prove:
	
	\begin{tr}\label{main}
		Let $R$ be a ring and $M$ a monoid. Assume $2n \geq \MD(R)$. Then ${\E}(2n,R[M])$ acts transitively on ${\Um}(2n,R[M]),$ in other words, 
		\[\frac{{\Um}(2n,R[M])}{{\E}({2n},R[M])} = \{*\} ~~({\rm trivial}).\]
	\end{tr}
	
	Next we obtain the above result in the relative case and establish surjective stabilization bounds for the relative Whitehead group $\widetilde{K}_1$ (where  $\widetilde{K}_1$ is either ${\ksp}$ or ${\ko}$).

	\begin{tr}\label{main2}
		Let R be a ring, $I \subseteq R$ an ideal and M a monoid. Then the canonical map 
		\[\varphi_{k}: \frac{{\G}(k,R[M], IR[M])}{{\E}(k,R[M], IR[M])} \rightarrow \widetilde{K}_1(R[M], IR[M])\]
		is surjective for $k \geq \MD(R)- 2$.
	\end{tr} 
	
	The reader may observe that using techniques in \cite{Gubeladze-unimodularMR3853049}, \cite{Maria-MKK1}, \cite{DhorajiaSymplectic-MR2971853}, \cite{Keshari-SymplecticMR2295082} and \cite{Bhatwadekar-SymplecticMR1858341} these results can be extended to non torsion-free monoids and for monoid extensions of certain overrings. Further, these results open up a whole host of questions on stabilization in the setup of other classical groups and possibly even for higher $K$-groups. \vp

	{\bf Layout of the paper:} 
	As has been the tradition with results in $K$-theory of monoid rings, owing to the connection monoid rings share with convex geometry, the paper will be a good mix of commutative algebra and convex geometry. We divide and discuss the preliminaries in two sections: The first summarizes the theory in rings that the reader needs to be familiar with and the second summarizes the theory of monoids and their corresponding rings. In Section 4, we will simplify the hypothesis by showing that it is enough to consider affine positive normal monoids. Section 5 is dedicated to proving Theorem \ref{main} and as its application Section 6 establishes surjective bounds for $\widetilde{K}_1$.

	\section{Classical Groups}\label{prelim}

	\noindent  Let $I \subseteq R$ be an ideal and $n>0$ a positive integer. Let $e_i$'s  denote the standard basis vectors with $1$ in the $i$'th position and $0$ elsewhere, and $e_{ij}(\lambda)$'s the matrix $\lambda e_ie_j^T$.
	By ${\Um}(n,R,I)$ one denotes the set consisting of elements $r=(r_1, \ldots, r_n) \in R^n$ satisfying the following:
	\begin{enumerate}[label=(\roman*), noitemsep]
		\item $r=e_1$ \text{ modulo } $I$ and
		\item $\exists$ $s=(s_1, \ldots, s_n) \in R^n$ such that $\smashoperator[r]{\sum_{i=1}^{n}} r_is_i = 1$.
	\end{enumerate}
	When $I=R,$ instead of ${\Um}(n,R,R)$ we use ${\Um}(n,R)$. To define the traditional classical groups, {\it viz.} symplectic and orthogonal groups, we 
	consider the following forms (resp.) on $R^2$, $ \phi_1 = \begin{pmatrix}
		0 & 1 \\
		-1 & 0 
	\end{pmatrix}$ and $\psi_1 = \begin{pmatrix}
		0 & 1 \\
		1 & 0 
	\end{pmatrix}$. On iterating, we may define the forms on $R^{2n}$:
	\[ \phi_n = \phi_{n-1} \perp \phi_1 
	\text{ and  } \psi_n = \psi_{n-1} \perp \psi_1.\]
	These forms help  define an inner products on $R^{2n}$ given by:
	\[\langle u,v \rangle:= \left\{
	\begin{array}{ll}
		u^T\phi_n v, & \text{ in the symplectic case, } \\
		u^T\psi_n v, & \text{ in the orthogonal case. }\\
	\end{array} 
	\right. \]
	\begin{re}\tn{
			When dealing with results in the setup of orthogonal groups, we will assume $char(R) \neq 2$ and any unimodular element $v$ is required to have an extra property of being isotropic, {\it i.e.}, $\langle v,v\rangle = 0.$}
	\end{re}
	
	Next we define the classical groups of even sizes: 
	
	\begin{de}
		The symplectic group $({\Sp}_{2n}(R))$ and orthogonal group $({\O}_{2n}(R)),$ over $R,$ are given by:

		\[{\Sp}_{2n}(R)= \{ \alpha \in {\GL}_{2n}(R) \mid \alpha^T\phi_n\alpha = \phi_n \}\] 
		\[{\O}_{2n}(R) = \{ \alpha \in {\GL}_{2n}(R) \mid \alpha^T\psi_n\alpha = \psi_n \}\]  
	\end{de}
	\noindent Let ${\rm Id}_{2n}$ denote the identity matrix of size $2n$ and $\sigma$ be the product of transpositions given by $\sigma(2i) = (2i-1)$ and $\sigma(2i-1) = 2i,$ for $1 \leq i \leq k = 2n$, {\it i.e.}, $\sigma = (1,2)(3,4)\ldots(2n-1,2n).$
	For $\lambda \in R$ and $1 \leq i < j \leq k=2n$ with $\sigma(i) \neq j,$ one defines 
	\begin{align*}
		se_{ij}(\lambda) &= {\rm Id}_{2n}+\lambda e_{ij} - \lambda (-1)^{i+j}e_{\sigma(j)\sigma(i)}  \\
		oe_{ij}(\lambda) &= {\rm Id}_{2n}+\lambda e_{ij} -\lambda e_{\sigma(j)\sigma(i)}
	\end{align*} 
	
	For values of $n,i,j, \lambda$ as indicated above, ${\ESp}_{2n}(R)$ is the elementary subgroup of ${\Sp}_{2n}(R)$ generated by $e_{i(\sigma(i))}(\lambda)$'s and $se_{ij}(\lambda)$'s and similarly ${\EO}_{2n}(R)$ is the elementary subgroup of ${\O}_{2n}(R)$ generated by $oe_{ij}(\lambda)$'s. Observe, in the orthogonal case we do not get generators for the cases $i=\sigma(j)$. For simplicity, we use the uniform notation $ge_{ij}(\lambda)$ for these elementary generators in both cases. 
	We condense some notations and write:
	\[ 
	{\G}(2n,R):= \left\{
	\begin{array}{ll}
		{\Sp}_{2n}(R), & \text{in the symplectic case, } \\
		{\O}_{2n}(R), & \text{in the orthogonal case. }\\
	\end{array} 
	\right. 
	\]
	
	\[ 
	{\E}(2n,R):= \left\{
	\begin{array}{ll}
		{\ESp}_{2n}(R), & \text{in the symplectic case, } \\
		{\EO}_{2n}(R), & \text{in the orthogonal case. }\\
	\end{array} 
	\right. 
	\]

	Let $J \subseteq R$ be an ideal and $k=2n$. Then we may define the relative subgroups ${\G}(k,R,J)$ as follows: 
	\[{\G}(k,R,J) =\{\alpha \in {\G}(k,R) \mid \alpha \equiv {\rm Id}_k~mod~J\}.\]
	For an integer $k>0$, we denote by ${\E}(k,J)$ the subgroup of ${\E}(k,R),$ generated by $ge_{ij}(\lambda)$, where $\lambda$ varies over $J$. The relative subgroup ${\E}(k,R,J)$ is defined as the normal closure of ${\E}(k,J)$ in ${\E}(k,R)$.
	
	Given the definition above there arises a natural group action of the elementary subgroup ${\E}(k,R,J)$ on ${\Um}(k,R,J)$. For the sake of clarity, we fix this action as multiplication on the right, {\it i.e.}, for $\sigma \in {\E}(k,R,J)$ and $u \in {\Um}(k,R,J)$
	\[\sigma.u := u\sigma \in {\Um}(k,R,J).\]
	
	\noindent Let $u, v \in {\Um}(k,R,J)$. By writing 
	\[u \simm{{\E}(k,R,J)} v\]
	we mean that $u$ and $v$ are in the same  orbit induced by the action of ${\E}(k,R,J)$ on ${\Um}(k,R,J)$. Often when unambiguous, we may use $u \sim v$. We denote the orbit space as ${\Um}(n,R,J)/{\E}(n,R,J)$ and often write ${\Um}(n,R,J)/{\E}(n,R,J) = \{*\},$ to represent trivial orbit space. We define $\widetilde{K}_1$ as  ${\ksp}$ in the symplectic case and ${\ko}$ in the orthogonal case.
	
	\begin{pr}\label{p3}$($\cite{Sus-K-MR0469914}, \cite{Kop-MR497932}$)$\label{pl} Let $R$ be a ring with identity with Krull dimension $d$. Then the orbit space is trivial under the action of elementary subgroups on unimodular rows over $R[\pz{m} \oplus \nz{m'}]$ for $2n \ge \MD(R)$, {\it i.e.},
		\[\frac{{\Um}(2n,R[\pz{m} \oplus \nz{m'}])}{{\E}(2n,R[\pz{m} \oplus \nz{m'}]) }= \{*\}.\]   
		Further, if R is semi-local and $n \geq 2,$ then ${\Um}(2n,R)/{\E}(2n,R) = \{*\}$.
	\end{pr}

	\begin{lm} $(${\it cf.} \cite{MR1958377}, Lemma 5.1$)$
		\label{noeth} 
		Let $R$ be a ring and $s\in R$ a non-zerodivisor. 
		Then there exists a natural number 
		$k$ such that the homomorphism
		$($induced by the localization homomorphism $R \ra R_s)$ ${\E}(2n, s^kR) \ra {\E}(2n, R_s)$  is injective.  
	\end{lm}

	\section{Monoids}\label{s2}
	\noindent Now we shift our focus to the theory of monoids. Let $M$ be a commutative multiplicative monoid with identity 1. 
	Let $\mbb{R}_+$ be the set of non-negative real numbers, and $\mbb{Z}_+$ of non-negative integers. By $\intt(M)$ we denote the elements of $M$ in the interior of the cone generated by $M,$ $i.e.$,
	\[\intt(M) = M \cap \intt(\mathbb{R}_+{M}).\] 
	
	\noindent Let us denote by $M_* = \intt(M) \cup \{1\}$. Let ${\rm gp}(M)$ denote the group completion of $M$ and ${\rm U}(M)$ the group of units of $M$. The monoid $M$ is called {positive }, if ${\rm U}(M)$ is trivial. Further, $M$ is defined to be {cancellative} if for $x, y, z \in M,$ $xz=yz$ implies $x=y$ and {torsion free}, if $x, y \in M,$ $n \in \pz{}$ and $x^n = y^n,$ then $x = y$. 
	
	One can restate the characteristic property of cancellation in terms of a property held by the corresponding inclusion map $i: M \rightarrow {\rm gp}(M)$. Note that $M$ is cancellative if and only if the map $i$ is injective. In this case, the torsion-freeness of $M$ is equivalent to the torsion-freeness of ${\rm gp}(M)$. The {rank} of $M,$ denoted by rank$(M)$, is the dimension of the $\mathbb{Q}$-vector space ${\rm gp}(M) \otimes \mathbb{Q}$. We say, $M$ is {affine} if it is a isomorphic to a submonoid of $\nz{d}$ for some $d \in \pz{},$ and is finitely generated over $\mbb{Z}_+$, {\it i.e.}, there exists $m_1,\ldots, m_k\in M$ such that $M=\mbb{Z}_+\langle m_1, \ldots, m_k\rangle$. For the sake of convenience, when dealing with affine monoids we often favour additive operation on the monoid. For an affine monoid $M,$ the entity $\mathbb{R}_+(M)$, called the cone of $M$, is given by:
	\[C:=\mathbb{R}_+(M) = \{a_1m_1 + \cdots +a_km_k \mid k \in \pz{}, a_i \in \mathbb{R}_+ \text{ and } m_i \in M\}. \]
	
	The assumption on our monoids to be commutative, cancellative and torsion-free is significant since it allows to conveniently embed the monoid ring $R[M]$ in the Laurent polynomial ring $R[X_1^{\pm{1}}, \ldots, X_r^{\pm{1}}],$ where $r$ is the rank of $M$. In the interest of clarity, one may note that in such a setup ${\rm gp}(M)$ will be torsion-free and thus isomorphic to $ \nz{r}$. With this understanding we proceed to our next set of definitions.
	
	Let $M \subseteq N$ be an extension of monoids. Then $N$ is said to be integral over $M$, if for $x \in N,$ there exists $c \in \mathbb{N}$ such that $x^c \in M$. Much like in the theory of rings, we may define integral closure of $M$ in $N$ and say $M$ is normal if it is integrally closed in ${\rm gp}(M)$. The monoid $M$ is said to be seminormal, if for $x \in {\rm gp}(M)$, whenever $x^2, x^3 \in M,$ then $x \in M$. It follows right away that a normal monoid is seminormal. 
	The monoid $M$ is defined as
	{$\phi$-simplicial}, if there exists an embedding $i_M$ of $M$  in $\pz{r}$ such that $\pz{r}$ is integral over $i_M(M)$ ({\it cf.}\cite{Gubeladze1-bookMR2508056}). 
	We are now in a position to discuss important (convex) geometrical aspects of monoids relevant to this article.

	\begin{pr}\label{ts}
		$($\cite{Gubeladze1-bookMR2508056}, Proposition 2.40$)$ An affine monoid $M$ is seminormal if and only if $(F_i \cap M)_*$ is normal for all $F_i \in Faces(C)$.
	\end{pr} 
	Let $M \subseteq \pz{r}$ be an affine monoid. Corresponding to the cone $C (\subseteq \mathbb{R}^r)$ of $M,$ one defines the dual cone $C^*$ as
	\[C^* = \{\lambda \in \Hom(\mathbb{R}^r,\mathbb{R}) \mid \lambda(c) \geq 0 ~\forall c \in C\}. \]
	For $\alpha \in \intt(C^*)$ and $a>0$, define the hyperplane $\mathcal{H} = \{x\in \mathbb{R}^r \mid \alpha(x) = a\}$. For any point $m \in M,$ we may define $\phi(m)$ as the intersection of the ray in $\mathcal{H},$ {\it i.e.}, 
	\[\phi(m) = \mathbb{R}_+(m) \cap \mathcal{H}.\] 
	We define $\phi(M):=\mathcal{H} \cap C =\mathcal{H} \cap \mathbb{R}_+(M).$
	Then by (\cite{Gubeladze1-bookMR2508056}, Proposition 1.21), $\mathcal{H} \cap C$ is a polytope. An intuitive way of defining this polytope is to look for a hyperplane $\mathcal{H} \in \mathbb{R}^r$ such that $\mathbb{R}_+(\mathcal{H} \cap C) = C$. We may extrapolate this definition to define the polytope $\phi(N),$ for $N \subseteq M$. Given a polytope $Q \subseteq \phi(M),$ we may define an affine $M$-submonoid $`M(Q)$' by
	\[M(Q) = \{m \in M \smallsetminus \{0\} \mid \phi(m) \in Q \} \cup \{1\} \subseteq M. \]

	\begin{re} \tn{The cone $C$ carries with it a wealth of geometric properties and is also used to create a dictionary between the algebraic and the geometric theory of monoids. It can be seen that $M$ is positive if and only if $0$ is a vertex of $C$ (these cones are often called pointed cones in the literature). Interestingly, $M$ is $\phi$-simplicial if and only if $C$ is a simplicial cone. The properties of seminormality and normality too carry the same sort of characterization relating the said algebraic property to a geometric property of the cone $C$. For details, we refer \cite{Gubeladze1-bookMR2508056}. }
	\end{re}

	\begin{de}
		\tn{Let $M$ be an affine positive normal monoid of rank $r$. By $k(M)$ we denote the complexity of $M$, defined as the smallest integer corresponding to which there exists a $k(M)-1$-dimensional polytope $Q$ and a collection of elements $\{m_1, \ldots, m_{r-k(M)} \} \subseteq \phi(M)$ such that}
		\[\phi(M) = \conv(m_1, \ldots, m_{r-k(M)}, Q )\]
	\end{de}
	
	\noindent It is easy to check that $M$ is $\phi$-simplicial if and only if $k(M)=0$.
	
	\begin{de}
		{\rm For a normal monoid $M$, an element $m \in M$ is called an extremal generator of $M,$ if} 
	\end{de}
	\begin{enumerate}[label=$($\roman*$)$, noitemsep]
		\item $\phi(m)$ is a vertex of $\phi(M)$ and 
		\item $\pz{} \simeq \mbb{R}_+(m) \cap M.$
	\end{enumerate}

	\begin{re} \tn{The intuition behind extremal generators is to capture the vertices of the polytope $\phi(M)$ which will help us decompose our monoids into a more manageable form as in Definition \ref{pdec}. The following theorem is crucial to the inductive approach applied in proof of Theorem \ref{main}.}
	\end{re}
	
	The following is an algebraic variant of Definition 7.2 in \cite{Gubeladze-unimodularMR3853049}.
	
	\begin{de}
		An $R$-algebra $T \subseteq R[t_1, \ldots, t_r]=S$ is said to be $t_1$-tilted if for each monomial $m \in S,$ there exists $p_m \in \pz{},$ such that for $p \geq p_m,$ we have $t_1^pm \in T$. 
	\end{de}
	
	\begin{tr}\label{tac}$($\cite{Gubeladze-unimodularMR3853049}, Lemma 3.4 and Section 8.1$)$
		Let $M$ be an affine positive normal monoid with $\rank(M) \geq 2$ and $m \in M$ an extremal generator. Then there exists $r \in \pz{}$ and monoids $M_1,$ $M_2,$ containing $M$ such that the following is a cartesian diagram 
		\begin{figure}[H]
			\tikzset{ampersand replacement=\&}
			\begin{tikzcd}[row sep=1.5em, column sep = 1.5em]
				A_1 \arrow[rr] \arrow[dr] \arrow[dd] \&\&
				R[M] \arrow[dd, "\cap" near end ] \arrow[dr] \\
				\& A_2 \arrow[rr] \&\&
				R[M] \arrow[dd, "\cap"] \\
				R[\pz{r}] \arrow[rr,swap, "\hspace{0.5cm}\pi_1"] \arrow[dr,] \&\& M_1 \arrow[dr] \\
				\& \hspace{-4cm}R[t_1, t_1^{k_2}t_2, \ldots, t_1^{k_r}t_r] \simeq R[\pz{r}] \arrow[rr, "\pi_2"] \arrow[uu] \&\& M_2
			\end{tikzcd}
		\end{figure}
		for appropriate $k_2, \ldots, k_r \in \pz{}$ and $i=1,2$, where
		\begin{enumerate}[label=$($\roman*$)$, noitemsep]
			\item $A_i$'s are the respective patching and $A_2 \subseteq R[t_1, t_1^{k_2}t_2, \ldots, t_1^{k_r}t_r] \simeq R[t_1, \ldots,t_r]$ is a $t_1$-tilted algebra,
			\item The slanted arrows are induced ring embeddings, and
			\item The maps $\pi_i:R[\pz{r}] \rightarrow M_i$ are surjective with the property: $\pi_i\vert_R = Id_R$ and $\pi_i(t_1) = m$.
		\end{enumerate}
	\end{tr} 
	
	Corresponding to a polytope $P \in \mathbb{R}^r$ and point $m \in \mathbb{R}^r,$ we may construct a pyramid in $ \mathbb{R}^r$ given by $(P,m):=\conv(P,m),$ where $\conv(-)$ represents convex combinations of elements in the parenthesis. In this setup, $m$ will be called the apex of $(P,m)$ and $P$ its base. Using this special polytopal structure we decompose monoids in a way that suits us:
	
	\begin{de}\label{pdec} \tn{
			For a monoid $M$, a decomposition $M=M(\Delta) \cup M(\Gamma)$ is called a pyramidal decomposition with respect to the extremal generator $m \in M$, if }	
		\begin{enumerate}[label=$($\roman*$)$, noitemsep]
			\item $\phi(M) = \Delta \cup \Gamma$,
			\item $\Delta \cap \Gamma$ is a face of $\Gamma$ and
			\item $\Delta$ is the pyramid with apex $m$ given by $\Delta = (\Delta \cap \Gamma, v)$.
		\end{enumerate}
	\end{de}
	We are primarily interested in a non-trivial decomposition arising out of $\Delta$ and $\Gamma$ being equi-dimensional (non-degenerate decomposition). With respect to any degree, for $f \in R[M]$, let $\mathcal{H}(f) \in R[M]$ denote its highest degree (leading) term  and  $\mathcal{L}_R(f)$ the coefficient of $\mathcal{H}(f)$ in $R$. We may switch to $\mathcal{L}(f)$ if the context is clear. 
	
	\begin{de}\tn{
			Given this decomposition with respect to $`m$', choose the linear map $\delta$ vanishing on $H=\Gamma \cap \Delta,$ such that $\delta(H \cap \nz{r}) = \nz{}$, $\delta(\Gamma)\leq 0$ and $\delta(\Delta)\geq 0$. Set $\deg(x):=\delta(x),$ for $x \in M$. This leads to the definition of pyramidal degree on $R[M]:$ For $f = \sum r_im_i,$ we define
			$\deg(f) := max\{\delta(m_i)\}.$} 
	\end{de}
	
	An element $f \in R[M]$ is called monic in $m,$ if $\mathcal{H}(f)=um^c,$ for some unit $u \in R^{\times}$ and $c \in \mathbb{N}$. For details, we refer to \cite{Gubeladze1-bookMR2508056}, \cite{Gub-geo&al-MR1414169}, \cite{reid&robert-MR1813497} and \cite{Swan-andersonMR1144038}.
	
	\section{Prologue: Simplification of hypothesis}
	In this section we will show that to prove Theorem \ref{main}, it is sufficient to show
	$\frac{{\Um}(2n,R[M])}{{\E}(2n,R[M])} = \{*\}$, when $R$ is local and $M \in \mathcal{N},$ where $\mathcal{N}$ is as below. This is done primarily by observing unimodular rows through patching diagrams.
	We establish this in two steps, $i.e.$, via Proposition \ref{ttsoh1} and Theorem \ref{ttsoh2}. Consider the following notations:
	\begin{align*}
		\mathcal{S} &:= \text{ the set of all seminormal monoids},\\
		\mathcal{N} &:= \text{ the set of all affine positive and normal monoids}.
	\end{align*}

	\begin{lm}\label{susl} $(${\it cf.}\cite{Sus-K-MR0469914}$)$
		Let R be a ring, $s \in R$ a non-zerodivisor and $\sigma \in {\G}(2n,R_s)$ and $\delta_{ij}(X) = \sigma ge_{ij}(Xf)\sigma^{-1},$ for $i>1$ and $f \in R_s[X]$. Then there exists $t \in \mathbb{N}$ and $\tau(X) \in {\E}(2n,R[X],XR[X])$ such that 
		\[\tau(X)_h = \delta_{ij}(Xs^t) = {\sigma}se_{ij}(Xs^t){\sigma}^{-1}.\]
	\end{lm}
	

	The proposition below and the consequent corollary demonstrate that if given appropriate patching diagrams (as in the proofs of Proposition \ref{ttsoh1},  Theorem \ref{ttsoh2} and Lemma \ref{karoubi}) the burden of transitive action can be shifted from one corner of the diagram to another.
	
	The following proposition can be seen as the classical analogue  of Vorst's decomposition. ({\it cf.}\cite{Gub-geo&al-MR1414169}, Lemma 6.1). For the definitions of cartesian diagram and Karoubi square refer to (\cite{Gubeladze1-bookMR2508056}, Section 8.C) .
	
	\begin{pr}\label{vor}
		
		\noindent Consider the commutative diagram $(D)$ below:
		\begin{figure}[H]
			\[
			\begin{tikzcd}[row sep=2em, column sep=2.5em]
				A \arrow[rr,"f_1"] \arrow[dd,swap,"f_2"] &&
				A_1 \arrow[dd,"g_1"] \\
				&  \\
				A_2 \arrow[rr,swap, "g_2"] && A’ 
			\end{tikzcd}
			\]
			\vspace{-0.4cm}
		\end{figure}
		\noindent If the diagram is either a Karoubi square, or cartesian with either $g_1$ or $g_2$ surjective, then any $\alpha \in {\E}(2n,A')$ can be decomposed as $\alpha = g_1(\alpha_1)g_2(\alpha_2),$ where $\alpha_i \in {\E}(2n,A_i)$ for $i=1, 2$.
	\end{pr}

	\begin{proof}
		
		If the diagram (D) is cartesian with say $g_1$ surjective, then we may choose $\alpha_1$ as a lift of $\alpha$. Choose $\alpha_2={\rm Id}_{2n} \in {\E}(2n, {A_2})$ and therefore we have the required decomposition. If (D) is a Karoubi square, then assume
		\[ \alpha = \smashoperator[r]{\prod_{k=1}^{m}}ge_{{i_k}{j_k}}(\lambda_k),\]
		where $ge_{{i_k}{j_k}}(\lambda_k) \in {\E}(2n,A')$ are respective elementary generators, $\lambda_k \in A'$ and $A_2 = S^{-1}A$. Without loss of generality, one may assume that $\lambda_k \in (A_1)_{f_1(s)},$ for $s \in S$. Fix $s' = f_1(s)$ and $\sigma_{k} = \prod_{1 \leq r \leq k-1} ge_{{i_r}{j_r}}(\lambda_r)$ for $1 \leq r \leq m$. From Lemma \ref{susl} corresponding to $\sigma_{k}$ and $s',$ there exists $t_k \in \mathbb{N}$ and $\tau_k(X) \in {\E}(2n,A_1[X], XA_1[X])$ such that
		\[\tau_k(X)_{s'} = {\sigma_k}ge_{{i_k}{j_k}}(X(s')^{t_k}){\sigma_k}^{-1}.\]
		For $1 \leq k \leq m,$ let $t = {\rm gcd}\{t_k\}$ and $t_k = td_k$. Since $f_1$ is an analytic isomorphism, we have the isomorphism $A/s^t \simeq A_1/(s')^t$. Therefore, for $1 \leq k \leq m,$ we may choose $u_k \in A_1,$ $v_k \in A$ and $z_k \in \pz{}$ such that
		\[ \lambda_k = u_k(s')^{t} + \frac{f_1(v_k)}{(s')^{z_k}} \in {(A_1)}_{s'}. \]
		
		\noindent Let $a_k = u_ks'^t$ and $b_k = f_1(v_k)/s'^{z_k}$ for $1 \leq k \leq m$. One gets:
		\begin{align*}
			\alpha &=  \sigma_m{ge_{{i_m}{j_m}}(\lambda_m)} \\
			&= \sigma_m{ge_{{i_m}{j_m}}(a_m)}{ge_{{i_m}{j_m}}(b_m)}\\
			& =\bigg(\sigma_m{ge_{{i_m}{j_m}}(a_m)}\smashoperator[r]{\prod_{ k = m-1}^{1}}{ge_{{i_k}{j_k}}(-b_k)}\bigg)\bigg(\smashoperator[r]{\prod_{i=1}^{m}}{ge_{{i_k}{j_k}}(b_k)}\bigg) \\
			& =\bigg(\sigma_m{ge_{{i_m}{j_m}}(a_m)}\smashoperator[r]{\prod_{ k = m-1}^{1}}{ge_{{i_k}{j_k}}(-\lambda_k)}{ge_{{i_k}{j_k}}(a_k)}\bigg)\bigg(\smashoperator[r]{\prod_{i=1}^{m}}{ge_{{i_k}{j_k}}(b_k)}\bigg) \\
			&= \smashoperator[r]{\prod_{k=m}^{1}}{\sigma_k}ge_{{i_k}{j_k}}(a_k) {\sigma_k}^{-1}\smashoperator[r]{\prod_{k=1}^{m}}{ge_{{i_k}{j_k}}(b_k)}.
		\end{align*}
		Let $u'_k = (s')^{d}u_k$, where $d \gg \{z_k\}$ for all $k$. Hence by Lemma \ref{noeth} one gets the following desired decomposition: 
		\[\alpha = \smashoperator[r]{\prod_{k=m}^{1}}\tau_k(u'_k)\smashoperator[r]{\prod_{k=1}^{m}}{ge_{{i_k}{j_k}}(b_k)}. \qedhere\] 
	\end{proof}
	
	\begin{co}\label{cmain} Let the diagram $(D)$ be as in Proposition \ref{vor}. If $u \in {\Um}(2n,A_1)$ is such that $g_1(u) \sim e_1,$ then there exists $v \in {\Um}(2n,A)$ such that $f_1(v) \sim u$. 
	\end{co}
	
	\begin{proof}
		Let $\alpha \in {\E}(2n,A')$ be such that $g_1(u)\alpha = e_1$. Using the result above we may find $\alpha_i \in {\E}(2n,A_i)$ for $i =1, 2$ such that $\alpha = g_1(\alpha_1)g_2(\alpha_2)$. Let $u_1 = u\alpha_1 \in {\Um}(2n,A_1)$ and $u_2 = e_1\alpha_2 \in {\Um}(2n,A_2)$. Then using the patching diagram (D) we get a unimodular element $v \in {\Um}(2n,A)$. The proof may be concluded by observing $f_1(v) = u_1 = u\alpha_1$. 
	\end{proof}
	
	Owing to Corollary \ref{cmain}, we may replicate the proof of the following proposition as done in Lemma (\cite{Gubeladze-unimodularMR3853049}, Lemma 6.1). 
	
	\begin{pr}\label{ttsoh1}
		Let R be a ring and $2n \geq \mathbb{D}(R)$.  Assume $\frac{{\Um}(2n,R[N_*])}{{\E}(2n,R[N_*])} = \{*\}$ for each $N \in \mathcal{N}$. Then  $\frac{{\Um}(2n,R[M])}{{\E}(2n,R[M])} = \{*\},$ for each $M \in \mathcal{S}$.
	\end{pr}

	The reader may also note that the proof of the above theorem demonstrates that to prove transitive action in $R[M]$ it is sufficient to prove transitive action for interior monoid algebra $R[N_*]$ where $N$ corresponds to the faces of $\phi(M)$. (Just like proof of Lemma 6.1(b) of \cite{Gubeladze-unimodularMR3853049}). The next step further relaxes the condition of seminormality of monoids. Let $J(R)$ denote the Jacobson radical of $R$.
	
	\begin{lm}\label{nilr}
		Let $I \subseteq J(R)$ be an ideal of R and $n \geq 2$.
		If $\frac{{\Um}(2n,R/I)}{{\E}(2n,R/I)} = \{*\}$, then $\frac{{\Um}(2n,R)}{{\E}(2n,R)} = \{*\}$. 
	\end{lm}
	
	\begin{proof}
		Let $\sim$ denote reduction modulo $I$ and $u \in {\Um}(2n,{R})$. Choose $\widetilde{\sigma} \in {\E}(2n,\widetilde{R})$ such that $\widetilde{\sigma}(\widetilde{u}) = e_1$. Lift $\widetilde{\sigma}$ to $\sigma \in {\E}(2n,R)$. Then $u' := \sigma(u) = (1+i_1, i_2, \ldots, i_{2n}) \in R^n,$ where $i_t \in I$ for $1 \leq t \leq 2n$. Observe that $1+i_1 \in U(R),$ since $i_1 \in I \subset J(R)$.
		We induct on $n$. Let $n=2$.
		\begin{align*}
			u' &= (1+i_1, i_2, i_3, i_{4})  \hspace{3.2cm} (\text{ apply } ge_{13}(-(i_3-1)(1+i_1)^{-1}) ) \\
			&\sim (1+i_1, i'_2, 1, i_4) \hspace{3.3cm} (\text{ apply } ge_{24}(i_1)) \\
			&\sim (1, i_2, 1, i_4') \hspace{4cm} (\text{ apply } ge_{14}(-i'_4)) \\
			&\sim (1, i''_2, 1, 0) \hspace{4.135cm} (\text{ apply }  ge_{13}(-1))\\
			&\sim (1, i''_2, 0, 0) 
		\end{align*}
		
		In the symplectic case, on applying $e_{12}(-i''_2)$ we can prove the required. For the orthogonal setup, note that $(1, i''_2, 0, 0)$ being isotropic would mean $i''_2 = 0$.
		Let $n>2$, and
		\[v=(v_1, \ldots, v_{2n}):=u'ge_{1(2n)}(-(i_{2n})(1+i_1)^{-1})ge_{1(2n-1)}(-(i_{2n-1})(1+i_1)^{-1}).\]
		By the inductive hypothesis, it follows that for $\widetilde{v} = (v_1, \ldots, v_{2n-2}) \in R^{2n-2},$ there exists $\widetilde{\sigma} \in {\E}(2n-2,R)$ such that $\widetilde{v}\widetilde{\sigma} = e_1 \in R^{2n-2}$. Choose $\sigma = \widetilde{\sigma} \perp I_2 \in {\E}(2n,R)$. Then $v\sigma = e_1$.
	\end{proof}
	
	The following is a consequence of the graded local-global theorem due to Basu-Chakraborty (\cite{Rab-Kun}, Theorem 8.2) and it helps us simplify the hypothesis of the main theorem to the local case. The proof of this proposition comes through by following the same steps as in Corollary 7.4 of \cite{Gubeladze-umodrow1-MR1161570}.
	
	\begin{pr}\label{tlg}
		Let $R$ be a ring and $\displaystyle{A=\mathop{\oplus}_{\substack{i \geq 0}} A_i}$ a positively graded $R$-algebra. Let $u \in {\Um}(2n,A)$ be such that $u\vert_{A_0} = e_{2n}$. Then $u \simm{{\E}(2n,A)} v$ if and only if $u \simm{{\E}(2n,A_{\m})} v$ for every $\m \in {\rm Max}(A_0)$.
	\end{pr}

	\begin{de} \tn{
			Let $R \subseteq S$ be an extension of rings. The extension is called an elementary subintegral extension if there exists $x \in S$ such that $S=R[x],$ where $x^2, x^3 \in R$. The extension is called subintegral if it is the filtered limit of elementary subintegral extensions.}
	\end{de}

	\begin{tr}\label{ttsoh2}
		Let $R \subseteq S$ be a subintegral extension. Then $\frac{{\Um}(2n,R)}{{\E}(2n,R)} = \{*\}$ if and only if $\frac{{\Um}(2n,S)}{{\E}(2n,S)} = \{*\}$.
	\end{tr}
	\begin{proof}
		Curiously enough to prove the above theorem we need to establish 
		\begin{align}
			\frac{{\Um}(2n,Z[M])}{{\E}(2n,Z[M])} = \{*\},
		\end{align}
		where $M$ is a $\phi$-simplicial monoid and $n \geq 2$. This can be done by replicating the proof in \cite{Gub-seminormality-MR1824231} as is for the orthogonal or symplectic matrices. 
		Without loss of generality we may assume $S=R[x],$ where $x^2, x^2 \in R$. \vp 
		Let us define:
		\begin{enumerate}[label=$($\roman*$)$, noitemsep]
			\item $B:=R \cup \{x\}$ and $A = \nz{}[\{t_b\}_{b \in B}]$,
			\item $\pi_1: A \rightarrow S$ the natural surjection sending $t_b$ to $b$,
			\item $P$ is the patching of $(R,i_1)$ and $(A,\pi_1)$,
			\item $N=\pz{}[t_b \mid t^2_x \text{ divides } t_b \text{ for } b \in B ]$,
			\item Let $I$ be the $\nz{}$-submodule of $P$ spanned by $N -\{1\}$.
		\end{enumerate}
		Then by Step 2 of Theorem 2 in \cite{Gub-seminormality-MR1824231}, we have that $N$ is a filtered limit of affine $\phi$-simplicial monoids and $I \subseteq P$ is an ideal of $P$. Consider the following patching diagrams:
		\begin{figure}[H]
			\tikzset{ampersand replacement=\&}
			\begin{equation*}
				\begin{tikzcd}
					P \arrow[rr,""] \arrow[dd,swap,""] \&\&
					R  \arrow[dd,"i_1"] \\
					\&  \\
					A  \arrow[rr, "\pi_1"] \&\& S
				\end{tikzcd}{}\hspace{1cm}%
				\begin{tikzcd}
					\nz{}[N] \arrow[rr,""] \arrow[dd,swap,""] \&\&
					P \arrow[dd,"\pi_2"] \\
					\&  \\
					\nz{} \arrow[rr, "i_1"] \&\& P/I.
				\end{tikzcd}{}
			\end{equation*}
		\end{figure}
		\noindent where $i_j$'s are inclusion and $\pi_j$'s are surjection for $i=1,2$. By applying Corollary \ref{cmain}, we see that it is sufficient to prove transitive action corresponding to the rings $\nz{}[N]$ and $P/I$. We have transitive action in $\nz{}[N]$ due to (1) as $N$ is a filtered limit of $\phi$-simplicial monoids. Also $(P/I)_{red} = \nz{}[\{t_a\} \mid a \in A]$. By Proposition \ref{pl}, $\frac{{\Um}(2n,R[\nz{k}])}{{\E}(2n,R[\nz{k}])} = \{*\}$. Since by Lemma \ref{nilr} transitive action is preserved going modulo nilpotents, therefore the conclusion follows.
	\end{proof}

	The simplification of hypothesis we claimed at the beginning now follows from Proposition \ref{ttsoh1} and Theorem \ref{ttsoh2}, where the argument for $R$ being local follows from Theorem \ref{tlg}. Thus in the next section we may assume that $R$ is a local ring and $M$ is an affine positive normal monoid.

	\section{Transitive action of ${\E}(2n,R[M])$ on ${\Um}(2n,R[M])$}
	
	The reader may refer to Section \ref{prelim} and \ref{s2} for notations. We start this section with the following lemma; a variant of the well-known prime avoidance lemma.
	
	\begin{lm}\label{l7.1}
		$($\cite{Gubeladze-unimodularMR3853049}, Lemma 7.1$)$ Let R be a ring and $u=(u_1, \ldots, u_n) \in {\Um}(n,R)$ for $n \geq 2$. Further, let $Ru_1+Rr=R$. Then for any finite family of ideals $\mathcal{I} = \{I_j\}_{j=1}^k$ such that $I_1, \ldots, I_k \subset R,$ there exists an element $r' \in (u_2, \ldots, u_n)$ such that for all $c \in \pz{},$  one has $u_1 + r^cr' \notin I_j$ for all $1 \leq j \leq k$.
	\end{lm}

	\begin{pr}\label{p7.3}
		Let R be a local ring, $T \subseteq R[t_1, \ldots, t_r]$ a $t_1$-tilted algebra and $n \geq 2$. For every $u \in {\Um}(2n,T)$ with $u \vert_{t_1 = 0} \in {\Um}(2n,R),$ there exists $v = (v_1, \ldots, v_{2n}) \in {\Um}(2n,T)$ such that
		\begin{enumerate}[label=$($\roman*$)$, noitemsep]
			\item $v \simm{{\E}(2n,T)} u$ and
			\item $\hgt_{R[t_1,\ldots,t_r]}(v_1, \ldots, v_i) \geq i$ for all $i$.
		\end{enumerate}
	\end{pr}
	\begin{proof}
		Let $A = R[t_1, \ldots, t_r]$ and $u = (u_1, \ldots, u_{2n})$. By Proposition \ref{p3}, we may assume  without loss of generality $u_i\vert_{t_1=0} = 1$ $\forall ~ i$. To prove the height inequalities we assume an inductive stance \noindent $\mathbb{H}_t$ on $t$, where $ 1 \leq t \leq \min\{2n-2, d+r\}$: \vp 
		
		\noindent $(\mathbb{H}_t):$ There exists $v \in {\Um}(2n,A)$ such that
		\begin{enumerate}[noitemsep]
			\item $u \simm{{\E}(2n,T)} v$,
			\item $v_i$ and $t_1$ are comaximal for all $1 \leq i \leq 2n$ and
			\item $\hgt_A(v_1, \ldots, v_j) \geq j$ for all $j \leq t$.
		\end{enumerate}
		
		\noindent Let $t=1$. Since $u_1 = 1+t_1{u'},$ for some ${u'} \in T,$ therefore $\hgt_A(u_1) \geq 1$. Let $t \geq 2$. From $(\mathbb{H}_{t-1}),$ we can procure $v \in {\Um}(2n,T)$ such that $v_iA + t_1A = A$ for all $1 \leq i \leq 2n$. We split the proof into two parts:\vp 
		
		\noindent \textbf{Case 1: $t$ is even}: Since $t<2n-1,$ by Lemma \ref{l7.1}, corresponding to the collection $\{\mathfrak{p}_i\}$ of all minimal primes over $I = A(v_1, \ldots, v_{t-1})$, we can find
		\[p =  p_1v_1 + \cdots + p_{t-1}v_{t-1} + p_{t+1}v_{t+1} + \cdots + p_{2n}v_{2n} \in A,\]
		
		\noindent such that for all $c \in \mathbb{N}$ and $i,$ we have $v_{t} + {t_1}^cp \notin \mathfrak{p}_i$. Let $p' = p_{t+1}v_{t+1} + \cdots + p_{2n}v_{2n}$. Then we may conclude that $v_{t} + {t_1}^cp' \notin \mathfrak{p}_i,$ for all $i$. Since $T$ is $t_1$-tilted we may multiply ${p}_i$ with appropriate power in $t_1$ to safely assume that  for large enough $c$
		\[p't_1^c  = a_{t+1}v_{t+1} + \cdots + a_{2n}v_{2n} \in T(v_{t+1}, \ldots, v_{2n}).\]
		
		\noindent Let $\eta = \displaystyle{\mathop{\prod}_{w=t+1}^{ 2n}} ge_{(t-1)w}(\lambda
		a_{\sigma(w)}),$ where 
		\[ 
		\lambda:= \left\{
		\begin{array}{ll}
			\hspace{0.25cm}1 & \text{ (symplectic case) $w$ is even,} \\
			-1 & \text{ else. }\\
		\end{array} 
		\right. 
		\]
		\noindent  Then the $t$'th entry of $v' = v\eta \in A$ is given by $v'_t=(v\eta)_t = v_t +p't_1^c + \delta v_{t-1},$ for some $\delta \in T$. Since $v_{t} + {t_1}^cp' \notin \mathfrak{p}_i,$ for all $i,$ we may conclude the height requirement by employing $\mathbb{H}_{t-1}$ and observing  
		\begin{align*}
			\hgt_A(v'_1, \ldots, v'_t) &= \hgt_A(v_1, \ldots, v_t+p't_1^c)\\
			&=  \hgt_A(v_1, v_2, \ldots, v_{t-1}) + 1 \geq t. 
		\end{align*}

		\noindent \textbf{Case 2: $t$ is odd}: 
		The proof is similar to above if we replace the former $I$ by $I = A(v_1, \ldots, v_{t-1}, v_{t+1})$ and consider the collection of minimal primes over $I$.
	\end{proof}
	
	\begin{co}\label{c7.4}
		Let R be a ring, $A=R[t_1, \ldots, t_r],$ $T \subseteq A$ a $t_1$-tilted algebra and $2n \geq \MD(R)$. Let $u \in {\Um}(2n,T)$ be such that $u\vert_{t_1=0} \in {\Um}(2n,R)$. Then there exists $v \in {\Um}(2n,T)$ such that $u \simm{{\E}(2n,T)} v$ and $\mathcal{L}(v_{2n})=1$.
	\end{co}
	
	\begin{proof}
		We break the proof into two cases: \vp 
		
		\noindent $(1)$: \textbf{(Symplectic case)} Let $I = A(u_1, u_2, \ldots, u_{d+1})$. By Proposition \ref{p7.3}, we may assume $\hgt_A(I) \geq d+1$. Then the ideal generated by $\mathcal{L}(g)$ where $g$ varies over $ I,$ is a unit ideal. Let $f = f_1u_1 + \cdots + f_{d+1}u_{d+1},$ where $f_i \in A,$ be such that $\mathcal{L}(f)=1$. Let $c > \deg(u_i)$ for all $1 \leq i \leq 2n,$ be such that $\lambda_i = t_1^cf_i \in T$ for all $1 \leq i \leq d+1$. Note that here $t_1^cf = \lambda_1u_1 + \cdots + \lambda_{d+1}u_{d+1}$.\vp
		
		\noindent {Subcase (i)}: If $d+1<2n-1,$ define:
		\[v' = u\prod_{1 \leq t \leq d+1}\hspace{-0.3cm}se_{t(2n)}(\lambda_t).\]
		Then $v'_{2n} = u_{2n} + t_1^cf + g$ for some $g = tu_{2n-1} \in T(u_{2n-1})$. Let $v = v'e_{(2n-1)(2n)}(-t)$. For $c \gg 0,$ we thus have $\mathcal{L}(v_{2n}) = \mathcal{L}(u_{2n} + t_1^cf) =1$. \vp 
		
		\noindent Subcase (ii): If $d+1 = 2n - 1,$ then we proceed as before with a small modification, and set $v' = u\prod_{1 \leq t \leq d}se_{t(2n)}(\lambda_t)$. Then $v'_{2n} = u_{2n} + \lambda_1u_1 + \cdots + \lambda_{d}u_{d} + g$ for some $g = tu_{2n-1} \in T(u_{2n-1})$. Let $v = v'e_{(2n-1)(2n)}(-t+\lambda_{d+1})$. Then 
		\[\mathcal{L}(v_{2n}) = \mathcal{L}(u_{2n} + t_1^cf) =1.\]
		
		\noindent $(2)$:\textbf{(Orthogonal case)} Let $I = A(u_1, u_3, \ldots, u_{2d+1})$. By Proposition \ref{p7.3}, proceeding inductively one may prove that $\hgt_A(u_1, u_3, \ldots, u_{2d+1}) \geq d+1$. Define $f, c , \lambda_i$ as in the previous case corresponding to $I$ (with appropriate caution to the $u_i$'s involved). Since $2d+1<2n-1,$ we may define $v \in {\Um}(2n,T)$ as
		\[v = u\Bigg({\displaystyle\prod_{~\{t \text{ odd } \mid 1 \leq  t \leq 2d+1\}}\hspace{-1cm}oe_{t(2n)}(\lambda_t)}\Bigg).\]
		Then $\mathcal{L}(v_{2n}) = \mathcal{L}(u_{2n} + t_1^cf) = 1$.
	\end{proof}

	\begin{lm}\label{llocal}
		Let $R$ be a ring and $M = M(\Delta) \cup M(\Gamma)$ be a pyramidal decomposition corresponding to the extremal generator $m \in M$. Assume $u \in {\Um}(2n,R[M])$ is such that $\mathcal{H}(u_{2n}) = rm^c,$ for $n \geq 2$ and some $r \in U(R),$ corresponding to the associated pyramidal degree. Then $u_{\m} \sim e_1$ for every $\m \in {\M}({R[\Gamma]})$.
	\end{lm}
	
	\begin{proof}
		By Lemma 5.1 of \cite{Gubeladze-unimodularMR3853049},  since $u_{2n}$ is a monic in $m,$ the extension
		\[R[\Gamma]_{\m} \rightarrow R[M]_{\m}/(u_{2n}) \text{ is integral for each } \m \in {\M}({R[\Gamma]}). \]
		
		\noindent Since $R[M]$ is noetherian, we have $R[M]_{\m}/(u_{2n})$ is semi-local. Let $I = R[M]_{\m}( u_{2n})$. and denote by $``-",$ reduction modulo $I$.  By Proposition \ref{pl}, ${\bar{u}}_{\m} \sim e_1$.  We would like to show that transitive action remains invariant modulo $I$.
		
		Let $\bar{\sigma} \in {\E}(2n,R[M]_{\m}/I)$ be such that $\bar{u}\bar{\sigma} = e_1$. Lift $\bar{\sigma}$ to $\sigma \in {\E}(2n,R[M]_{\m})$ such that
		\begin{align*}
			u\sigma  = (u'_1, \ldots, u'_{2n-2}, u'_{2n-1}, u_{2n}),
		\end{align*}
		where $(u'_1, \ldots, u'_{2n-1}) \equiv e_1~mod~I$ and $u_i \in R[M]_{\m}$. Thus we can find appropriate $\tau \in {\E}(2n,R[M]_{\m})$ such that 
		\begin{align*}
			v\tau &= (1,0,0,0,\ldots,\alpha u_{2n} + \beta, u_{2n})\\
			&\sim (1,u_{2n}(\alpha u_{2n} + \beta),0,0,\ldots,0,u_{2n})\\
			&\sim (1,u_{2n}(\alpha u_{2n} + \beta),0,0,\ldots,0,0)
		\end{align*}
		for some $\alpha, \beta \in R[M]_{\m}$. In the symplectic setup, on applying $se_{12}(-u_{2n}(\alpha u_{2n} + \beta)),$ we get $u_{\m} \sim e_1$. For the orthogonal case by the isotropic property, we have
		\[u_{2n}(\alpha u_{2n} + \beta) = 0,\]
		indicating the required.
	\end{proof}

	\begin{pr}\label{plam}
		Let R be a ring and assume patching squares as in Theorem \ref{tac}. Let $M=M(\Delta) \cup M(\Gamma)$ a pyramidal decomposition w.r.t. the extremal generator $m \in M$.  If ${\Um}(2n,M_1)/{{\E}(2n, M_1)} = \{*\}$, then $u_{\m} \simm{{\E}(2n,R[M]_{\m})} e_1,$ for every $ u \in {\Um}(2n, R[M])$ and $ \m \in {\M}(R[M(\Gamma)])$.
	\end{pr}
	
	\begin{proof}
		Let $ u \in {\Um}(2n, R[M])$. By Lemma \ref{llocal}, it sufficient to show that
		\[u \sim u' = (u'_1, \ldots, u'_{2n}) \in {\Um}(2n,R[M]),\]
		such that $u'_{2n}$ is a monic with respect to the associated pyramidal degree. 
		
		Since $u \in {\Um}(2n,R[M]) \subseteq {\Um}(2n,M_1),$ by Corollary \ref{cmain},  there exists $v \in {\Um}(2n,A_1)$ such that $\pi_1(v) = u$. Given the slanted arrows in the patching diagram in Theorem \ref{tac} are embeddings and 
		\[A_2 \subseteq R[t_1, t_1^{k_2}t_2, \ldots, t_1^{k_r}t_r] \simeq R[t_1, \ldots,t_r]\]
		is a $t_1$-tilted algebra, we have $v\vert_{t_1=0} \in {\Um}(2n,R)$. We may assume $\mathcal{L}(v_{2n}) \in U(R),$ by Corollary \ref{c7.4}. For $k \gg 0,$ define a simplex $\Delta,$ generated as a convex combination of $\phi(t_1)$ and $\phi(t_1^kt_i)$'s for $2 \leq i \leq r$.\vp 
		
		Observe that $\pz{r} \simeq \pz{r}(\Delta) \subseteq \pi_2^{-1}(M),$ and hence by Proposition \ref{pl}, 
		\[v \simm{{\E}(2n,R[\pz{r}(\Delta)])} e_1.\]
		Choose $\sigma \in {\E}(2n,R[\pz{r}(\Delta)])$ such that $v = e_1\sigma$. For any $c \in \pz{},$ one may define $\beta = \sigma^{-1}\tau(\sigma),$ where $\tau$ is the Nagata endomorphism on $R[t_1, \ldots, t_r]$ defined by
		\[\tau(t_j) \mapsto t_j + t_1^c \text{ for } 2 \leq j \leq r.\]
		
		As $\pz{r}(\Delta) \subseteq \pi_2^{-1}(M)$ and $A_2 = R[\pi_2^{-1}(M) + ker(\pi_2)],$ we have $R[\pz{r}(\Delta)] \subseteq A_2$. Therefore for large values of $c,$ $\beta = \sigma^{-1}\tau(\sigma) \in {\E}(2n,A_2)$. 
		
		By Proposition \ref{pl} and (\cite{Basu-Rao-MR2578583}, Lemma 3.6) there exists $\sigma' \in {\E}(2n,\pz{r}(\Delta))$ and $\beta' \in {\G}(2n-2,R[\pz{r}(\Delta)])$ such that 
		\[ \beta = \sigma'\begin{pmatrix}
			\beta' & 0 & 0 \\
			0 & 1 & 0 \\
			0 &0 & 1
		\end{pmatrix}\]
		
		\noindent Define \begin{align*}
			w = (w_1, \ldots, w_{2n-1}, \tau{(v_{2n})}) : = \tau(v)&\begin{pmatrix}
				\beta'^{-1} & 0 & 0 \\
				0 & 1 & 0 \\
				0 &0 & 1
			\end{pmatrix}\\
			= v\beta& \begin{pmatrix}
				\beta'^{-1} & 0 & 0 \\
				0 & 1 & 0 \\
				0 &0 & 1
			\end{pmatrix} = v\sigma' \simm{{\E}(2n,A_2)} v.
		\end{align*} 
		Since $\pi_2(t_1)=m,$we choose $c\gg0,$ so that $\pi_2(\tau(v_{2n}))$ is a monic in $m$ and thus we may conclude the proof by observing $u \simm{{\E}(2n,R[M])} \pi_2(w).$
	\end{proof}
	\noindent Given a monoid $M$ with complexity $k(M),$ henceforth we appoint as $Q,$ the $k(M)-1$ dimensional polytope such that
	\[\phi(M) = \conv(m_1, \ldots, m_{\rank(M)-k(M)}, Q).\]
	\noindent For an polytope $Q' \subseteq Q,$ define:
	\[\widetilde{Q'} := \conv(m_1, \ldots, m_{\rank(M)-k(M)}, Q')\]

	\begin{lm}\label{karoubi}
		Let $(R, \mu)$ be a local ring and Theorem \ref{main} be true for $R[N],$ for $N \in \mathcal{N}$ with $k(N) < k(M)$. For $Q$ as above, if $P \subseteq Q$ is a polytope with the decomposition $M(P) = M(\delta) \cup M(\gamma)$. Then for $u \in {\Um}(2n,R[M(\widetilde{P})_*]),$ there exists $v \in {\Um}_n(R[M(\widetilde{\gamma})_*]),$ such that $u \simm{{\E}(2n,R[M(\widetilde{P})_*])}v$.
	\end{lm}
	
	\begin{proof}
		Let $z \in \intt(\widetilde{\gamma})$ and $\theta_c(z)$ denote the homothetic transformation about $z$ with radius $c \in (0,1) \cap \mathbb{Q}$. We may choose a $c$ such that $u \in {\Um}(2n,R[M(\theta_c(z)(\widetilde{P}))])$. Let $P' = \theta_c(z)(\widetilde{P})$ and $\gamma' = \theta_c(z)(\widetilde{\gamma})$.
		
		By Proposition 8.4 of \cite{Gubeladze-unimodularMR3853049} 
		and Proposition \ref{plam} yields $u_{\m} \simm{{\E}(2n,R[M(P')]_{\m})} e_1$ for $\m = (\mu, M(\gamma')-\{0\}) \in {\M}{(R[M(\gamma')])}$ (though implicit here, this is the step where the complexity restriction of the hypothesis is used). Let $\mathfrak{n} = (\mu, M(\widetilde{\gamma}) - \{0\}) \in \Max(R[M(\widetilde{\gamma})])$. Using Lemma 8.2 in \cite{Gubeladze-unimodularMR3853049}, we have the Karoubi square
		\begin{figure}[H]
			\[
			\begin{tikzcd}[row sep=2em, column sep=2.5em]
				R[M(\gamma)_*] \arrow[rr,"i_1"] \arrow[dd,swap,""] &&
				R[M(P)_*] \arrow[dd,"\pi_i"] \\
				&  \\
				R[M(\gamma)_*]_{\mathfrak{n}} \arrow[rr,swap, "i'_1"] &&R[M(P)_*]_{\mathfrak{n}}.
			\end{tikzcd}
			\]
		\end{figure}
		\noindent From Corollary \ref{cmain} owing to the inclusion 
		\[R[M(P')]_{\m} \subseteq R[M(\widetilde{P})_*]_{\mathfrak{n}},\]
		there exists $u \in R[M(\gamma)_*]$ such that $v \simm{{\E}(2n,R[M(P)_*])} u$.
	\end{proof}
	
	\noindent Now our main result follows seamlessly using induction on $k(M)$. \vp 
	
	\textit{Proof of Theorem \ref{main}:}
	\noindent As observed in Proposition \ref{ttsoh1}, the burden of transitive action shifts to the interior monoid algebra $R[M_*].$ Let $u \in {\Um}(2n,R[M_*])$.
	We prove this theorem by inducting on $k:=k(M)$. Consider the induction hypothesis for $0 \leq t \leq k$:\vp  
	
	\noindent $(\mathbb{H}_t):$ If $N \in \mathcal{N}$ and $t=k(N)$, then
	\[\frac{{\Um}(2n,R[N_*])}{{\E}(2n,R[N_*])}= \{*\} \]
	
	\noindent For the base case ($\mathbb{H}_0$) observe that if $N \in \mathcal{N}$ with $k(N)=0,$ then $N$ is a simplex. Consequently, $N_*$ is a filtered limit of $\phi$-simplicial monoids, and we are done by the first half of the proof of Theorem \ref{ttsoh2}. \vp
	
	\noindent Let $k>0$. Consider $Q,$ the $k-1$ dimensional polytope such that
	\[\phi(M) = \conv(m_1, \ldots, m_{\rank(M)-k}, Q).\]
	Let $\Delta$ be a simplex such that $\intt(\Delta) \subseteq Q$ and $u \in {\Um}(2n,R[M(\widetilde{\Delta})_*])$. Corresponding to this neighbourhood, by Lemma 8.3 of \cite{Gubeladze-unimodularMR3853049}, we may find a sequence of polytopes $\{Q_i\}_{i \in \pz{}} \subseteq Q$ and a stage $i,$ such that for $j \geq i$ we have $Q_i \subseteq \Delta$. If the hypothesis of Lemma \ref{karoubi} is satisfied, then there exists $v \in {\Um}(2n,R[M(\widetilde{Q_i})_*])$ such that \[ u \simm{{\E}(2n,R[M(\widetilde{Q_i})_*])} v.\] Without loss of generality we may assume $Q_{i} = \Delta$. By observing that $\widetilde{Q_i}$ is again a simplex, we use $\mathbb{H}_0$ to get $v \simm{{\E}(2n,R[M(\widetilde{Q_i})_*])} e_1$. Now we prove the conditions required to use Lemma \ref{karoubi}. \vp
	
	\textit{ Claim}: Let $R$ be a local ring, then Theorem \ref{main} is true for $R[N],$ where $N \in \mathcal{N}$ with $k(N)<k=k(M)$.\vp
	
	\noindent If $F$ is a face of $\phi(N),$ then $k(M(F))\leq k(N)<k$ and therefore by induction 
	\[\frac{{\Um}(2n,R[M(F)_*])}{{\E}(2n,R[M(F)_*])}= \{*\}.\]
	The remark following Proposition \ref{ttsoh1} implies the indicated claim.
	\qed

	\section{Surjective Stabilization of $\k$-group}
	
	The surjective stabilization of the $\k$ functor of classical groups is a natural consequence of the transitive action of their elementary subgroups on unimodular rows. In this section, using a trick for the excision ring, we deduce a relative version of the main result. Then, as a consequence, we find the surjective stabilization bound for the respective $\k$-groups.

	\begin{de}
		\tn{Given an ideal $I \subseteq R,$ one defines the excision ring $R \oplus I$ with the following operations}
		\begin{align*}
			(r_1, i_1)+(r_2, i_2) &= (r_1 + r_2, i_1 + i_2),\\
			(r_1, i_1)\cdot (r_2, i_2) &= (r_1r_2, r_1i_2+r_2i_1+i_1i_2),
		\end{align*}
		where $r_j \in R$ and $i_j \in I$ for all $i$.
	\end{de}
	There exists a natural homomorphism $\phi: R \oplus I \rightarrow R,$ given by $\phi(r,i) = r+i$. Also $\dim(R \oplus I) = \dim(R)$ and, $R \oplus I$ is noetherian, if $R$ is noetherian. 
	We state some results that are required.	
	
	\begin{lm}\label{le}
		(\cite{Basu-Pillar}, Lemma 2.13, 2.14) Let R be a commutative ring and $I \subset R$ an ideal. If $\alpha \in {\E}(k, R, I),$ then there exists $\widetilde{\alpha} \in {\E}(k,R \oplus I)$ such that $\phi(\widetilde{\alpha}) = \alpha$. Further, the converse holds too.
	\end{lm}
	
	\begin{lm}\label{lem}
		Let $I' \subseteq R$ be an ideal and $I = I'R[N],$ where $N$ is a monoid. Then 
		\[R[N] \oplus I \simeq (R \oplus I')[N].\]
	\end{lm}
	
	\begin{proof}
		Define $\phi: R[N] \oplus I'R[N] \rightarrow (R \oplus I')[N]$ as $\phi(\sum r_jn_j, \sum i'_jn_j) = \sum (r_j, i'_j)n_j$. (We may choose the same index as we are in a finite setup!)
	\end{proof}

	\begin{tr}\label{rel}
		Let $R$ be a ring, $I \subseteq R$ an ideal and $M$ a monoid. Then the following orbit space is trivial for $2n \geq \MD(R):$ 
		\[\frac{{\Um}(2n,R[M],I)}{{\E}({2n},R[M],I)} = \{*\}.\]
	\end{tr}
	
	\begin{proof}
		Let $A = R[M]$ and $v=(v_1, \ldots, v_{2n}) \in {\Um}(2n,R[M],I)$. Define $v' =  (v_1 + 0, \ldots, v_n+0) \in {\Um}(2n,R[M] \oplus I)$. By Lemma \ref{lem}, $A \oplus I = (R \oplus I)[M],$ where $R \oplus I$ is a $d$-dimensional noetherian ring.

		Let $\phi$ be the natural homomorphism $\phi: A \oplus I \rightarrow A$. On application of Theorem \ref{main}, we may choose $\sigma' \in {\E}(2n, A \oplus I)$ such that $\sigma'v' = (1 + 0, 0 +0, \ldots, 0+0) \in {\Um}(2n, A \oplus I)$.  Then $\phi(\sigma') = \sigma \in {\E}(2n,A, I),$ by the converse of Lemma \ref{le}. Then $e_1 = \phi(\sigma'v') = \phi(\sigma')(\phi(v')) = \sigma(v)$.
	\end{proof}
	
	As a consequence of this, we have Theorem \ref{main2}:

	\begin{tr}
		Let R be a ring, $I \subseteq R$ an ideal and M a monoid. Then the canonical map 
		\[\varphi_{k}: \frac{{\G}(k,R[M], IR[M])}{{\E}(k,R[M], IR[M])} \rightarrow \widetilde{K}_1(R[M], IR[M])\]
		is surjective for $k \geq \MD(R)- 2$.
	\end{tr}
	
	\begin{proof}
		For $G \in {\G}(r, R[M]),$ consider the following composition of maps
		\begin{align*}
			{\G}(r,R[M], IR[M]) &\rightarrow \frac{{\G}(r,R[M], IR[M])}{{\E}(r,R[M], IR[M])} \rightarrow \widetilde{K}_1(R[M], IR[M])\\
			G  &\longmapsto \bar{G} \longmapsto [G]
		\end{align*} 
		Let $k=\mathbb{D}(R)-2$ and $[S] \in \widetilde{K}_1{(R[M])}$. We require $T \in {\G}(k,R[M], IR[M])$ such that $[S] = [T]$ in $\widetilde{K}_1(R[M], IR[M])$. Without loss of generality we may assume $S \in {\G}(k+2, R[M], IR[M])$. One may note that the last column of $S$ belongs to ${\Um}(k+2, R[M], IR[M])$. By Theorem \ref{rel}, there exists $\theta \in   {\E}(k+2, R[M], IR[M])$ such that 
		\[S\theta e_{k+2} = e_{k+2} \]
		By (\cite{Basu-Rao-MR2578583}, Lemma 3.6) there exists $\sigma' \in {\E}(k+2,R[M])$ and $T \in {\G}(k,R[M])$ such that 
		\[ S\theta = \sigma'\begin{pmatrix}
			T & 0 & 0 \\
			0 & 1 & 0 \\
			0 &0 & 1
		\end{pmatrix}\]
		Since $S\theta$ belongs to ${\G}(k,R[M], IR[M]),$ this in turn compels the decomposed components to also belong to the relative groups, i.e, $\sigma' \in {\E}(k+2,R[M], IR[M])$ and $T \in {\G}(k,R[M], IR[M])$. 
		The proof can be concluded by observing $[T] = [S]$ in $\widetilde{K}_1(R[M])$.
	\end{proof}
	
	\section*{Declarations}

	\begin{itemize}
		\item Funding: Research by the first author was supported by the Centre Franco-Indien pour la Promotion de la Recherche Avancée (CEFIPRA) grant for the period November 2020-- November 2023. Research by the second author was supported by the Indian Institute of Science Education and Research (IISER) Pune post-doctoral research grant for the year 2023.
		
		\item Competing interests: The authors have no conflict of interest to declare that are relevant to this article.

		\item Authors' contributions: All authors have contributed to this work and have read and approved the final manuscript.
		
		\item Acknowledgments: The second author would like to thank Prof. Manoj Keshari for introducing them to the area of monoid rings. The authors would like to dedicate this paper to  Prof. Ravi A. Rao on the occasion of his 70th birthday.
		
		
	\end{itemize}

	\bibliographystyle{abbrv}
	\bibliography{ref}
	
\end{document}